\documentclass[12pt,reqno]{amsart}
\usepackage{amssymb,amsmath,amsthm,pdfsync}

\usepackage{amssymb,color}
\usepackage{mathrsfs}
\usepackage{amsmath}
\usepackage{amsthm}
\usepackage{url} 
\usepackage{stackrel} 
\usepackage{courier,amscd} 
\usepackage[makeroom]{cancel}
\usepackage[all,knot]{xy} 
\usepackage{multicol} 
\usepackage{verbatim} 
\usepackage{mathtools}  
\usepackage{hyperref}
\usepackage{enumerate}  
\usepackage{bookmark} 
\usepackage{ytableau}
\usepackage{microtype}
\usepackage{bbold}  
\usepackage{tikz}
\usetikzlibrary{decorations.markings}
  \usepackage{graphicx}
\usepackage[all]{xy}

\catcode`~=11 
\newcommand{\urltilde}{\kern -.15em\lower .7ex\hbox{~}\kern .04em}  
\catcode`~=13 

 \xyoption{arc}

\newcommand{\Mor}{\mathop{\mathrm{Mor}}\nolimits}

\newcommand{\sgn}{\mathop{\mathrm{sgn}}\nolimits}

\newcommand{\proj}{\mathop{\mathrm{proj}}\nolimits}
\newcommand{\Ind}{\mathop{\mathrm{Ind}}\nolimits}

\newcommand{\sign}{\mathop{\mathrm{sign}}\nolimits}

\newcommand{\im}{\mathop{\mathrm{Im}}\nolimits}

\newcommand{\Res}{\mathop{\mathrm{Res}}\nolimits}

\newcommand{\Ob}{\mathop{\mathrm{Ob}}\nolimits}

\newcommand{\aff}{\mathop{\mathrm{aff}}\nolimits} 
\newcommand{\Mod}{\mathop{\mathrm{Mod}}\nolimits} 
 
\newcommand{\Hom}{\mathop{\mathrm{Hom}}\nolimits}

\newcommand{\End}{\mathop{\mathrm{End}}\nolimits}

\newcommand{\eqnref}[1]{~(\ref{#1})}

\newcommand{\germ}[1]{\mathrm{#1}}

\makeatletter
\newif\if@borderstar
   \def\bordermatrix{\@ifnextchar*{%
       \@borderstartrue\@bordermatrix@i}{\@borderstarfalse\@bordermatrix@i*}%
   }
   \def\@bordermatrix@i*{\@ifnextchar[{\@bordermatrix@ii}{\@bordermatrix@ii[()]}}
   \def\@bordermatrix@ii[#1]#2{%
   \begingroup
     \m@th\@tempdima8.75\p@\setbox\z@\vbox{%
       \def\cr{\crcr\noalign{\kern 2\p@\global\let\cr\endline }}%
       \ialign {$##$\hfil\kern 2\p@\kern\@tempdima & \thinspace %
       \hfil $##$\hfil && \quad\hfil $##$\hfil\crcr\omit\strut %
       \hfil\crcr\noalign{\kern -\baselineskip}#2\crcr\omit %
       \strut\cr}}%
     \setbox\tw@\vbox{\unvcopy\z@\global\setbox\@ne\lastbox}%
     \setbox\tw@\hbox{\unhbox\@ne\unskip\global\setbox\@ne\lastbox}%
     \setbox\tw@\hbox{%
       $\kern\wd\@ne\kern -\@tempdima\left\@firstoftwo#1%
         \if@borderstar\kern2pt\else\kern -\wd\@ne\fi%
       \global\setbox\@ne\vbox{\box\@ne\if@borderstar\else\kern 2\p@\fi}%
       \vcenter{\if@borderstar\else\kern -\ht\@ne\fi%
         \unvbox\z@\kern-\if@borderstar2\fi\baselineskip}%
         \if@borderstar\kern-2\@tempdima\kern2\p@\else\,\fi\right\@secondoftwo#1 $%
     }\null \;\vbox{\kern\ht\@ne\box\tw@}%
   \endgroup
   }
\makeatother

\newtheorem{theorem}{Theorem}[section] 
\newtheorem{proposition}[theorem]{Proposition}
\newtheorem{lemma}[theorem]{Lemma}

\newtheorem{definition}[theorem]{Definition}
\newtheorem{remark}[theorem]{Remark} 

\newtheorem{corollary}[theorem]{Corollary}

\allowdisplaybreaks

\begin{document}

\title[Categorification of Verma and indecomposable projective modules for $\mathfrak{sl}_2$]{Categorification of Verma modules and indecomposable projective modules in the category $\mathcal I_{\mathfrak{g}}(\mathfrak{sl}_2)$ for $\mathfrak{sl}_2$}  
\author{ 
Ben Cox \and 
Mee Seong Im 
 }  
\address{Department of Mathematics, University of Charleston, Charleston, SC 29424 USA}
\email{coxbl@cofc.edu}
\address{Department of Mathematical Sciences, United States Military Academy, West Point, NY 10996 USA}
\email{meeseongim@gmail.com}
\date{\today}

\begin{abstract}  
We categorify Verma and indecomposable projective modules in the category $\mathcal I_{\mathfrak{g}}(\mathfrak{sl}_2)$ for $\mathfrak{sl}_2$ using a tensor product decomposition theorem of T. J. Enright and work of J. Chuang and R. Rouquier \cite{MR2373155}, A. Licata and A. Savage \cite{MR3032820} and  M. Khovanov \cite{khovanov2010heisenberg}. 
 
\end{abstract}   
\thanks{Both authors thank the Department of Mathematics at the University of California at Santa Cruz for providing conducive working environment in May-June of 2015. 
MSI would like to thank the College of Charleston for their hospitality in September of 2015 and BC would like to thank the United States Military Academy for their hospitality in October of 2015. 
BC is partially supported by Simons Collaboration Grant $\#319261$ and MSI is supported by NSF-AWM Mentoring Grant.} 


\maketitle 
\bibliographystyle{amsalpha}   
\setcounter{tocdepth}{2}

\section{Introduction}\label{section:introduction}  

We answer Rapha\"el Rouquier's question of categorifying Verma modules in the setting of $\mathfrak{sl}_2(\mathbb{C})$. That is, we obtain \textit{ $\mathcal{I}_{\mathfrak{g}}(\mathfrak{sl}_2)$-categorifications} $\mathcal V_s$ and $\mathcal T_r$ of Verma modules $V_s$ for $s$ an integer and their indecomposable projective covers $T_r$, where $r$ is a nonnegative integer. We define our notion of an $\mathcal{I}_{\mathfrak{g}}(\mathfrak{sl}_2)$-categorification below.  Verma modules play a fundamental role in the representation theory of Kac-Moody Lie algebras and in the geometry of infinite partial flag variety. Verma modules also correspond to $\delta$-function distributions on orbits under the action by the Borel and to certain holomorphic functions on these orbits (cf. \cite{MR2838836}, \cite{MR2357361}, \cite{MR797420}, \cite{MR1260747}, \cite{MR1260748}), so the categorification of these modules play a fundamental role in the theory of categorification in the Lie algebra setting. 

Naisse and Vaz used techniques of \cite{MR2373155}, \cite{MR2776785}, and \cite{MR2305608} to produce a geometric approach using infinite Grassmannians and infinite partial flag varieties to categorify Verma modules for $U_q(\mathfrak{sl}_2)$ 
(cf. \cite{Naisse-Vaz-Verma}).

The motivation of our construction goes as follows: a classical result in the representation theory of $\mathfrak{sl}_2$ states that any Verma module $V_\lambda$ of highest weight $\lambda\in\mathbb C$ can be realized in terms of operators $x,\partial_x$ of the Heisenberg algebra acting on the polynomial ring $\mathbb C[x]$ with one indeterminate $x$. For this reason, the Heisenberg algebra is also realized as a Weyl algebra, or a ring of differential operators. 
Now let $ H_{\mathbb Z}$ be the infinite-rank Heisenberg algebra defined over nonnegative integers $\mathbb N$ with generators $1$ and $p_i,q_i$, where $i\in \mathbb N$, together with relations $p_i q_j =q_j p_i + \delta_{ij}$.  M. Khovanov in \cite{khovanov2010heisenberg} constructs a conjectured categorification $\mathcal H$ of the infinite-rank Heisenberg algebra $H_{\mathbb Z}$ and shows the existence of a map $\rho$ from $\mathcal H$ to the category $\text{End}(\oplus_{n\in\mathbb N}\mathbb C[S_n]$-$\text{mod})$ of functors on the direct sum of categories of $\mathbb C[S_n]$-modules (more recently, a proof of the surjectivity of the ring homomorphism from $H_{\mathbb{Z}}$ to $K_{0}(\mathcal{H})$ has been announced in \cite{categorificationWWYY}).  $\mathcal H$ is the Karoubi envelope of the category $\mathcal H'$ defined as a strict monoidal category generated by two objects $Q_+$ and $Q_-$ and morphisms between tensor products of such objects are given by planar diagrams mod out by certain relations.  The category of endofunctors can be viewed as a type of categorification of a Fock space.  Under the appropriate assignment of functors:  
\begin{center}
$\rho(F)=Q_+$, $\hspace{4mm}$ $\rho(E)= - Q_{+} (Q_{-})^2  $, 
\end{center} 
which correspond to $\rho(f)=x$ and 
$\rho(e)=-x\partial_x^2$, we arrive at the weak $\mathcal{I}_{\mathfrak{g}}(\mathfrak{sl}_2)$-categorification (cf. Definition~\ref{definition:modified-weak-categ}) of the highest weight Verma module with highest weight $0$.   We denote this categorification by $\mathcal V_0$.  
The objects of $\mathcal{V}_0$ are sums of the functors obtained from  
$$
\Lambda^n_{+} := (Q_{+^n},e'(n)), \hspace{2mm}
S_{-}^n := (Q_{-^n}, e(n)), \hspace{2mm} 
e(n)=\frac{1}{n!}\sum_{\sigma\in S_n}\sigma,\hspace{2mm}  
e'(n)=\frac{1}{n!}\sum_{\sigma\in S_n}\sign(\sigma)\sigma \in \mathbb k [S_n], 
$$
and the morphisms (which are viewed as natural transformations) are certain planar diagrams modulo relations given by Khovanov.  
In the $U_q(\mathfrak{sl}_2)$-geometric categorification construction in \cite{Naisse-Vaz-Verma}, $F$ is only a left adjoint of $E$  (they are not biadjoints). Similarly in our construction, $F$ is only a left adjoint of $E$.  

For a complex Lie algebra $\mathfrak g$ containing $\mathfrak{sl}_2$, T. J. Enright defines the categories $\mathcal I(\mathfrak{sl}_2)$ and $\mathcal{I}_{\mathfrak{g}}(\mathfrak{sl}_2)$ as follows: 
\begin{definition}[\cite{MR541329}, Definition 3.2]\label{defn:category-I-sl2}
We define $\mathcal{I}(\mathfrak{sl}_2)$ to be the category of $\mathfrak{sl}_2$-modules such that the following properties hold:  
\begin{enumerate} 
\item $M$ is a weight module for $\mathfrak{sl}_2(\mathbb C)$, with $M=\oplus_{\mu\in\mathbb C}M_\mu$ where $M_\mu=\{m\in M : h.m=\mu m\}$, 
\item $f$ acts injectively on $M$, 
\item $e$ acts locally nilpotently on $M$.  
\end{enumerate}
\end{definition}
\begin{definition}[\cite{MR541329}, page 8]\label{defn:I-g-sl2-category}
For any Lie algebra $\mathfrak{g}$ over $\mathbb{C}$ containing $\mathfrak{sl}_2$, we define $\mathcal I_{\mathfrak g}(\mathfrak{sl}_2)$ as the category of $\mathfrak{g}$-modules whose underlying $\mathfrak{sl}_2$-module lies in $\mathcal{I}(\mathfrak{sl}_2)$. 
\end{definition} 

\begin{theorem}
Let $M$ be an $\mathfrak{sl}_2$-module in $\mathcal{I}_{\mathfrak{g}}(\mathfrak{sl}_2)$. Then 
\begin{enumerate} 
\item if $M$ is indecomposable, then $M\cong V_{\lambda}$ for some $\lambda \in \mathbb{C}$ or $M\cong T_{\lambda}$ for some $\lambda \in \mathbb{N}_{\geq 0}$,    
\item $M$ is an (infinite) direct sum of indecomposable modules. 
\end{enumerate}
\end{theorem}

\begin{remark}[\cite{MR541329}, Proposition 3.11.(i)]
The Verma module $V_{\lambda}$ ($\lambda \in \mathbb{C}$) and the projective covers $T_{\lambda}$ ($\lambda \geq 0$ an integer) are precisely all the indecomposable objects of the category $\mathcal{I}(\mathfrak{sl}_2)$. 
\end{remark}

 Let $L_n$ denote the irreducible highest weight module for $\mathfrak{sl}_2$ of highest weight $n$ of dimension $n+1$ and let $V_{\lambda}$ be the Verma module with highest weight $\lambda$, where $\lambda \in\mathbb C$.  \color{black} In the annals paper by Enright (\cite{MR541329}, Proposition 3.12), the author gives a tensor product decomposition for $L_n\otimes V_{\lambda}$ in $\mathcal I_{\mathfrak{g}}(\mathfrak{sl}_2)$:  
\begin{equation}\label{Enrightsresult} 
L_n\otimes V_\lambda\cong \bigoplus_{r\in I'}T_{\lambda+r}\oplus\bigoplus_{s\in I'''}V_{\lambda+s}, 
\end{equation}
where $I'$ and $I'''$ are explicitly described finite sets of integers (see Proposition~\ref{prop:direct-sum-decomp-Enright}) and $T_{\lambda+r}$ are indecomposable projective covers of Verma modules $V_{\lambda+r}$.  
We categorify a version of this decomposition to construct our categorification of Verma modules for $\mathfrak{sl}_2$.

The final piece of motivation comes from the work of J. Chuang and R. Rouquier on the $\mathfrak{sl}_2$-categorification $\mathcal L_n$ of the modules $L_n$.  These categories have endofunctors $E$ and $F$ defined on them that decategorify to the action of $e$ and $f$ on $L_n$. Moreover there are natural transformations $X=\oplus_i X_i\in \text{End}(F)$ and $T=\oplus_i T_i\in \text{End}(F^2)$ such that $X_i$ and $T_i$ satisfy functorial versions of the relations for the affine Hecke algebra. 

Putting the above together, we define the category $\mathcal{L}_n\boxtimes \mathcal{V}_0$ in such a manner so that its objects are of the form $\displaystyle{\oplus_{k,l}} M_k\boxtimes_{} N_l$, where $M_k$ is an object of $\mathcal{L}_n$ and $N_l$ is an object of $\mathcal{V}_0$, and morphisms are of the form $\sum_{k,l} f_k\otimes g_l$, where $f_k\in \Hom_{\mathcal{L}_n}(M_k,M_k')$ and $g_l\in \Hom_{\mathcal{V}_0}(N_l,N_l')$.    We then define the categories $\mathcal T_r$ and $\mathcal V_s$, where $r\in\mathbb N$  and $s\in \mathbb Z$, as certain subcategories of $\mathcal L_n\boxtimes \mathcal{V}_0$.  As a consequence there are functors $E$ and $F$ defined on $\mathcal L_n\boxtimes \mathcal{V}_0$ that leave the subcategories $\mathcal T_r$ and $\mathcal V_s$ stable.  We prove that these subcategories are categorifications of the modules $T_r$ and $V_s$, respectively, whereby the functors $E$ and $F$ decategorify to the action of $e$ and $f$ on $\mathfrak{sl}_2$-modules $T_r$ and $V_s.$
Categories $\mathcal{V}_s$ and $\mathcal{T}_r$ have as objects sums of submodules generated by applications of $E$ and $F$. 

We now state one of the main results in this paper: 

\begin{theorem}\label{theorem:weak-sl2-categorification} 
Let $I=I'\cup I''\cup I'''$ be as above. 
    For $n\geq 0$  and $\lambda\in\mathbb Z$, we have a direct sum decomposition of categories
    \begin{align*}
 \bigoplus_{r\in I'} \mathcal T_{r}\oplus\bigoplus_{s\in I'''}  \mathcal V_{s},    
\end{align*}
which is a full subcategory of $\mathcal L_n\boxtimes \mathcal V_0 $,
where $\mathcal L_n\boxtimes \mathcal  V_0$ is the category whose objects are sums of tensor products of objects of the two categories $\mathcal L_n$ and $\mathcal V_0$ and whose morphism likewise are sums of tensor products of morphisms.
In particular, we obtain $\mathcal{I}_{\mathfrak{g}}(\mathfrak{sl}_2)$-categorifications $\mathcal V_s$ and $\mathcal T_r$ of Verma modules $V_s$ and their projective covers $T_r$, respectively, where $s\in I'''$ and $r\in I'$.
\end{theorem}

  Here the tensor product category $\mathcal L_n\boxtimes \mathcal  V_0$  is made into an $\mathcal{I}_{\mathfrak{g}}(\mathfrak{sl}_2)$-categorification (see Section~\ref{section:weak-categorifications}) using comultiplication of 
  affine Hecke algebra, compatible with the actions of $E$ and $F$, given as a Lie algebra action. We also construct natural transformations 
 $X\in \End(F)$ and $T\in \End(F^2)$ acting on the tensor product of categories $\mathcal{L}_n$ and $\mathcal{V}_0$, compatible with the (nondegenerate and degenerate) affine Hecke algebra structure.

For $s\in I'$, $r=-s-2$, and $0\leq i\leq (n+r)/2$, define 
\begin{equation}
p_i = 4^{\frac{n+r}{2} - i } \dfrac{n+r+2}{n+r-2i+2} \prod_{j=0}^{i-1} (n+r-2j)^2 \prod_{\nu=i}^{\frac{n+r-2}{2} } (n-\nu),  
\end{equation}
where if $i=0$, we define 
\[ 
\prod_{j=0}^{i-1} (n+r-2j)^2 := 1. 
\] 
Let $v_j\in L_n$, a finite dimensional $\mathfrak{sl}_2$-irreducible representation with highest weight $n$. 
Then 
\[ 
v_j = \dfrac{1}{j!} f^j v_0, \mbox{ where } 0 \leq j\leq n  
\] 
satisfying: 
\begin{equation}
f.v_i = (i+1)v_{i+1}, \hspace{4mm} 
e.v_i = (n-i+1)v_{i-1}, \hspace{4mm} \mbox{ and } \hspace{4mm}    
h.v_i = (n-2i)v_i.  
\end{equation}
Now, we make an identification $V_0 \cong \mathbb{C}[x]$ of the $\mathfrak{sl}_2$-Verma module $V_0$ with highest weight $0$ with the polynomial ring $\mathbb{C}[x]$, let $w_k:= x^k \in \mathbb{C}[x]$. 
Then 
\begin{equation}
f.w_k = w_{k+1}, \hspace{4mm}
e.w_k = -k(k-1)w_{k-1}, \hspace{4mm}  \mbox{ and } \hspace{4mm}  
h.w_k = -2kw_k.  
\end{equation} 
Consider the highest weight vector $u_n:= v_0 \otimes w_0 \in L_n \boxtimes V_0$ of highest weight $n$. 
We restate the second part of Theorem~\ref{theorem:weak-sl2-categorification} as follows: 
\begin{theorem}\label{theorem:sl2-verma-module-categorification} 
Let $s\in \mathbb{Z}$.  
Let $K_{\oplus}( \mathcal{V}_s)$ be the split Grothendieck group of the category of sums of submodules generated by applications of $E$ and $F$.  
Then 
\begin{equation} 
\mathbb{C} \otimes_{\mathbb{Z}} K_{\oplus}(\mathcal{V}_s) \stackrel{\simeq}{\longrightarrow} V_s 
\hspace{4mm}
\mbox{ sending }
\hspace{4mm} 
 [U_s]\mapsto u_s,  
\end{equation}  
where 
\[ 
U_s = \boxplus_{j=0}^{\frac{n-s-2}{2} } (M_j \boxtimes N_{\frac{n-s}{2} - j})^{\boxplus p_j} 
\hspace{4mm}
\mbox{ and } 
\hspace{4mm}
u_s = \sum_{j=0}^{\frac{n-s-2}{2} }p_j v_j\otimes w_{\frac{n-s}{2} -j},   
\]  
which is a weight vector of weight $s$. 
Furthermore, if 
$s$ is a nonnegative integer, then 
\[ 
\mathbb{C} \otimes_{\mathbb{Z}} K_{\oplus}(\mathcal{V}_s) \stackrel{\simeq}{\longrightarrow} V_s 
\hspace{4mm}
\mbox{ also sends } 
\hspace{4mm}
[U_{-s-2}] \mapsto u_{-s-2},   
\]   
where  
\[ 
U_{-s-2} = \boxplus_{j=0}^{\frac{n+s}{2}} (M_j \boxtimes N_{\frac{n+s+2}{2}-j})^{\boxplus p_j}
\hspace{4mm}
\mbox{ and } 
\hspace{4mm}
u_{-s-2} = f^{s+1}u_s = \sum_{j=0}^{\frac{n+s}{2}} p_j v_j \otimes w_{\frac{n+s+2}{2}-j} ,  
\]  
with weight $-s-2$.
The isomorphism of split Grothendieck groups also respects a map of functors: 
\[ 
[F:=Q_+\otimes -]\mapsto f=x 
\hspace{4mm}
\mbox{ and }
\hspace{4mm}
 [-E:=Q_+Q_-Q_-\otimes -]\mapsto -e=-x\partial_x^2. 
\]  
\end{theorem} 

In particular, $\mathbb{C} \otimes_{\mathbb{Z}} K_{\oplus} (\mathcal{V}_0)\stackrel{\simeq}{\longrightarrow} V_0\cong \mathbb{C}[x]$, where the $\mathfrak{sl}_2$-action on $V_0$ is the action by the noncommutative algebra $\mathbb{C}\langle x,\partial_x\rangle /\langle \partial_x x-x\partial_x -1\rangle$.  

The projective cover 
$T_r$ is generated by $\{f^ka_{-r-2}: k\in \mathbb{N} \} \cup \{ f^k u_r: k\in \mathbb{N}\}$, where 
\begin{equation}\label{eq:gen-of-Tr-a-r-2}
a_{-r-2} = \sum_{j=0}^{\frac{n+r}{2}} q_j v_j \otimes w_{\frac{n+r+2}{2} - j } 
\hspace{4mm}
\mbox{ and } 
\hspace{4mm}
u_r = \sum_{j=0}^{\frac{n-r-2}{2}}p_j v_j \otimes w_{\frac{n-r}{2}-j} ,  
\end{equation} 
where the $q_i$'s satisfy the recurrence relation in Equation~\eqref{equation:recursion-equation-basis-coeff-betas}.  
\begin{theorem}\label{theorem:sl2-projective-module-categorification} 
Let $r$ be a nonnegative integer. 
We have an isomorphism of modules: 
\begin{equation}
\mathbb{C} \otimes_{\mathbb{Z}} K_{\oplus} (\mathcal{T}_r)\stackrel{\simeq}{\longrightarrow}T_r 
\hspace{4mm}
\mbox{ sending }
\hspace{4mm}
[A_{-r-2}]\mapsto a_{-r-2} \mbox{ and } [U_r]\mapsto u_r, 
\end{equation}
where 
\[ 
A_{-r-2} = \boxplus_{j=0}^{\frac{n+r}{2}} (M_j \boxtimes N_{\frac{n+r+2}{2} - j })^{\boxplus q_j}  
\hspace{4mm}
\mbox{ and } 
\hspace{4mm}
U_r =  \boxplus_{j=0}^{\frac{n-r-2}{2} } (M_j \boxtimes N_{\frac{n-r}{2} - j})^{\boxplus p_j} ,   
\] 
and $a_{-r-2}$ is a generator of the submodule $T_r$ given in Equation~\eqref{eq:gen-of-Tr-a-r-2} and $u_r$ is a generator in $T_r$ isomorphic to $V_r$. 
This isomorphism also gives a map of functors. 
\end{theorem}

\begin{remark}
We emphasize on the positivity of the coefficients $p_i$ and $q_j$ in Theorem~\ref{theorem:sl2-verma-module-categorification} and Theorem~\ref{theorem:sl2-projective-module-categorification}. 
\end{remark}

In Section~\ref{subsection:decomp-tensor-product-modules}, we recall that all Verma modules $V_{\lambda}$ for $\mathfrak{sl}_2$ for the highest weight $\lambda\in \mathbb{C}$ can be realized using the Heisenberg algebra action on the polynomial ring $\mathbb{C}[x]$ in one indeterminant. 
In the case when $\lambda=0$, Khovanov has a categorification $\mathcal H$ of the Heisenberg algebra $H_{\mathbb{Z}}$ with infinitely-many generators $p_i$ and $q_i$, with $i\in \mathbb N_{>0}$, which ``acts'' on the category $\text{End}(\oplus_{n\geq 0}S_n$-mod$)$ of functors whose Grothendieck group is the Fock space with highest weight $\lambda=0$.

The categorification using degenerate affine Hecke algebras has been constructed by Khovanov in \cite{khovanov2010heisenberg}, 
while nondegenerate affine Hecke algebras were used in \cite{MR3032820}. 

For us $\mathcal{L}$ is an additive category, 
$\mathcal{T}_r$ is a category of projective modules and $\mathcal{V}_0$ is an additive category. 
$\mathcal{T}_r$ has objects that are indecomposable, but $\mathcal{T}_r$ is not triangularizable.  Thus by a result of Adelman, the abelianization of $\mathcal T_r$ is zero. 

\subsection{Adelman's abelianization} Here we review the construction of an abelian category from an additive one using the technique due to M. Adelman
\cite{MR0318265} as it deserves to be better known.

Suppose $\germ{M}$ is an additive category and then consider the category $\germ{M}^{\to\to}$ whose objects are double arrows
\begin{equation}
A'\to A\to A''
\end{equation}
where $A'$, $A$, and $A''$ are all objects in $\germ{M}$ and the horizontal maps are morphisms in $\germ{M}$. 
The morphisms in the category $\germ{M}^{\to\to}$ are triples $x',x,x''$ where the diagram 
$$
\begin{CD}
A' @>>> A @>>> A'' \\
@Vx'VV @VxVV @VVx''V \\
B' @>>> B @>>> B'' \\
\end{CD}
$$
is commutative. 

In this category one defines an equivalence relation as follows:
$$
\begin{CD}
A' @>a'>> A @>a>> A'' \\
@V\alpha'VV @V\alpha VV @VV\alpha ''V \\
B' @>>b'> B @>>b> B'' \\
\end{CD}
$$
is equivalent to 
$$
\begin{CD}
A' @>a'>> A @>a>> A'' \\
@V\beta 'VV @V\beta VV @VV\beta ''V \\
B' @>>b'> B @>>b> B'' \\
\end{CD}
$$
if there exists morphisms $s_1:A\to B'$ and $s_2:A''\to B$ such that a homotopy condition is satisfied 
\begin{equation}
b's_1+s_2a=\alpha-\beta. 
\end{equation}

The quotient category $\germ{M}^{\to\to} /\equiv$ is denoted by $\text{Ab}(\germ{M})$. 

The kernel of
$$
\begin{CD}
A' @>a'>> A @>a>> A'' \\
@V\alpha'VV @V\alpha VV @VV\alpha ''V \\
B' @>>b'> B @>>b> B'' \\
\end{CD}
$$
is 
$$
\begin{CD}
A' \oplus B'@>\phi >> A @>\psi>> A'' \\
@V(1,0)VV @V(1,0)VV @VV(0,1),V \\
B' @>>b'> B @>>b> B'' \\
\end{CD}
$$
where
$$
\phi=\begin{pmatrix} a' & 0 \\ \alpha ' & 1 \end{pmatrix},\quad\text{ and }\quad \psi =\begin{pmatrix} \alpha & -b \\ a & 0 \end{pmatrix}.
$$

The cokernel of $(\alpha',\alpha,\alpha'')$ is 
$$
\begin{CD}
B' @> >> B @>>> B'' \\
@VVV @VVV @VVV \\
B'\oplus A @>>\gamma> B \oplus A''@>>\rho> B''\oplus A'', \\
\end{CD}
$$
where
$$
\gamma=\begin{pmatrix} b' & \alpha \\ 0 & -a \end{pmatrix},\quad\text{ and }\quad \rho =\begin{pmatrix} b & \alpha''\\ 0 & -1 \end{pmatrix}.
$$

One then identities the category $\germ{M}$ with the full subcategory of $\text{Ab}(\germ{M})$ consisting of double arrows of the form $0\to A\to 0$ where $A$ is an object in $\germ{M}$.

\begin{theorem}[\cite{MR0318265}, Theorem 1.14] Suppose $F:\germ{M}\to \germ{B}$ is an additive functor with $\germ{B}$ being an abelian category. Then there exists a unique exact functor $K:\text{Ab}(\germ{M})\to \germ{B}$ such that the diagram 
\[ 
\xymatrix@-1pc{
\ar[rdd]_{F} \germ{M} \ar[rr]^{I_{\germ{M}}}  & & \text{Ab}(\germ{M}) \ar[ldd]^{K}\\ 
 & & \\ 
 &\germ{B} &  \\ 
}
\] 
is commutative up to natural equivalence (the functor $K$ is also unique up to natural equivalence).
\end{theorem}
In the above, $I_{\germ{M}}$ is defined on objects as $I_{\germ{M}}(A)$ is the equivalence class of $0\to A\to 0$ and $I_{\germ{M}}(A\to B)$ is the equivalence class of
$$
\begin{CD}
0@> >> A @>>> 0 \\
@VVV @VVV @VVV \\
0 @>>> B @>>> 0.\\
\end{CD}
$$

We leave it as an exercise to show for Enright's category that $\text{Ab}(\mathcal I)=0$, and thus $\text{Ab}(\mathcal T)=0$.

The category 
$\mathcal{V}_0$ has objects $Q_+$, which are induction (graphical) functors, and morphisms $X$ and $T$.

\section{Background}\label{section:background}

\subsection{Modules and their bases}\label{subsection:decomp-tensor-product-modules}
The idea behind our construction of the categorification of a Verma module for $\mathfrak{sl}_2$ goes back to the tensor product decomposition formula of Enright for modules in category $\mathcal I_{\mathfrak{g}}(\mathfrak{sl}_2)$. 

Let $L_n$ be the $n+1$ dimensional $\mathfrak{sl}_2$-module with highest weight $n$.   

\subsubsection{Finite dimensional $\frak{sl}_2$-modules $L_n$}
The commutator identity $[e,f^k]=kf^{k-1}(h-k+1)$ is used to prove the following: 
\begin{lemma}[\cite{MR499562}] The finite dimensional $\mathfrak{sl}_2$-module $L_n$ with highest weight $n$ has basis $\{ v_i\}_{i=0}^n$, where $v_j = \dfrac{1}{j!} f^j v_0$ for $0\leq j\leq n$. 
In particular,  
\begin{equation}
 f.v_i = (i+1)v_{i+1}, \hspace{4mm} 
 e.v_i= (n-i+1)v_{i-1}, \hspace{4mm}  
h.v_i = (n-2i)v_i. 
\end{equation}

\end{lemma}

\subsubsection{Verma modules $V_\lambda$ for $\frak{sl}_2$}  Let $U(\mathfrak{g})$ denote the enveloping algebra of a Lie algebra $\mathfrak{g}$, $\mathfrak h=\mathbb Ch$ its Cartan subalgebra, $\mathfrak b=\mathbb Ch\oplus \mathbb Ce$ is Borel subalgebra and $\mathbb C_\lambda=\mathbb C v_\lambda$ 
the one dimensional $\mathfrak b$-module defined by $h. v_{\lambda}=\lambda v_{\lambda}$, with $e. v_{\lambda}=0$ for a fixed $\lambda\in \mathbb C$. 
Let $V_\lambda =U(\mathfrak{sl}_2)\otimes_{U(\mathfrak b)}\mathbb C_\lambda$ be the Verma module with highest weight $\lambda\in\mathbb C$.   Then by the Poincar\'e-Birkhoff-Witt Theorem,  
$$
V_\lambda= \mathbb C[f]v_\lambda=\bigoplus_m V_{\lambda,\lambda-2m}, 
$$
where $v_\lambda$ is the highest weight vector with weight $\lambda$ and 
$$
V_{\lambda,\lambda-2m}=\mathbb Cf^mv_\lambda=\{w\in V_\lambda:  h.w=(\lambda-2m)w\},
$$
the $\lambda-2m$ weight space.

Next, Lemma~\ref{lemma:Sophus-Lie} is a classical result due to Sophus Lie. 
\begin{lemma}\label{lemma:Sophus-Lie}  
	Consider the ring $\mathbb C[x]$ as a vector space over $\mathbb C$ and let $\lambda\in\mathbb C$. We define $\rho:\mathfrak{sl}_2\to \text{End}(\mathbb C[x])$ as follows: 
		\begin{align*}
			\rho(f)&:=x, \\
			\rho(e)&:=-x\partial^2_x+\lambda \partial_x, \\ 
			\rho(h)&:=-2x\partial_x+\lambda. \\ 
		\end{align*}
	Then $\rho$ defines an $\mathfrak{sl}_2$-representation isomorphic to $V_\lambda$.
\end{lemma}

In particular the Verma module $V_0\cong \mathbb{C}[x]$, with basis $w_k$, where $w_k=x^k$, with 
\begin{equation}
\rho(f)(w_k)=f.w_k = w_{k+1}, \hspace{1mm}  
\rho(e)(w_k)=e.w_k = -k(k-1)w_{k-1},    \hspace{1mm} 
\rho(h)(w_k)=h.w_k = -2kw_k. 
\end{equation} 

 \subsubsection{Indecomposable projective modules $T_r$. }    The Casimir operator for $\mathfrak{sl}_2$ is defined to be the element
\begin{equation}
\Omega = h^2+2h+4fe, \label{casimir}
\end{equation}
in $U(\mathfrak{sl}_2)$.
For $r \in \mathbb N_{>0}$ consider the following left ideals of $U(\mathfrak{sl}_2)$:
\begin{align}
\label{idealM} \mathfrak{M}_r &:=U(\mathfrak{sl}_2)\{h+n+2, e^{r+2}, (\Omega_{\sl}-r^2-2r)^2\},
\end{align}
and define
\begin{align}
T_r&:=U(\mathfrak{sl}_2)/\mathfrak{M}_r.
\end{align}
Let $1_r$ denote the image of 1 in   $T_r$.  Then $T_r$ is an indecomposable projective module in the category $\mathcal I_{\mathfrak{g}}(\mathfrak{sl}_2)$ (see \cite{MR541329}) and the structure of the $\mathfrak{sl}_2$-module $T_r$ has the form:   
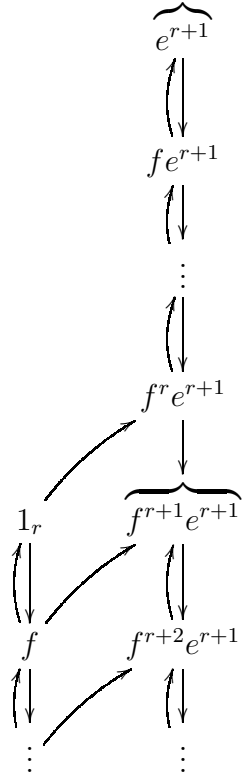
\[
\xymatrix{
& {\overbrace{e^{r+1}}} \ar[d]\\
& f e^{r+1} \ar@/^/[u] \ar[d]\\
& {\vdots} \ar@/^/[u] \ar[d]\\
& f^r e^{r+1} \ar@/^/[u] \ar[d]\\
{1_r} \ar@/^/[ur] \ar[d] & {\overbrace{f^{r+1}e^{r+1}}} \ar[d]\\
f \ar@/^/[ur] \ar@/^/[u] \ar[d] & f^{r+2} e^{r+1} \ar@/^/[u] \ar[d]\\
{\vdots} \ar@/^/[ur] \ar@/^/[u] & {\vdots} \ar@/^/[u] 
}
\]
Each vertex above constitutes a weight vector in a basis for $T_r = U(\mathfrak{sl}_2)/\mathfrak{M}_r$ and upward arrows (including the upward arrows at $45$ degrees) denote the action of $e$ and downward arrows denote the action of $f$ on a basis vector.    The right most column, which includes weight vectors $f^{i} e^{r+1}$, where $0\leq i\in\mathbb{Z}$, is a submodule of $T_r$ which is isomorphic to $V_{r}$ and weight vectors along the left column are isomorphic to $V_{-r-2}$ after one quotients out $T_r$ by the submodule isomorphic to $V_{r}$.   Upward facing braces indicate that the vector below them is a highest weight vector.  The weight vector $1_r$ has weight $-r-2$ and is a generator of $T_r$. We thus have a short exact sequence: 
\[ 
0 \rightarrow V_r \rightarrow T_r \rightarrow V_{-r-2} \rightarrow 0,  
\] 
and the action by $\mathfrak{sl}_2$ on weight vectors in $T_r$ is described as follows: 
\[ 
\begin{aligned}
f(f^ie^j) &= f^{i+1}e^j, \\ 
e(f^ie^j) &= f^i e^{j+1} + (2ij -i^2-i(n+1))f^{i-1}e^j,   \\ 
h(f^ie^j) &= (2j-2i-n-2) f^i e^j. \\ 
\end{aligned} 
\]

\subsubsection{Tensor product of a standard and a Verma module}

Proposition~\ref{prop:direct-sum-decomp-Enright} from \cite{MR541329} (cf. Proposition 3.12) gives us a decomposition of the tensor product of a standard module and a Verma module into a decomposition of irreducible modules. 

\begin{proposition}\label{prop:direct-sum-decomp-Enright}  Let $\lambda \in\mathbb Z$ and $0\leq n\in\mathbb N$. The weights of $ L_n$ are in the set 
$I=\{-n,-n+2,\dots,n-2, n\}$, and define sets as follows: 
\begin{align*}
I'&=\{r\in I :  \lambda+r\geq 0\text{ and }-(\lambda+r)-2-\lambda\in I\},  \\
I''&= \{-r-2\lambda -2 :  r\in I'\}, \mbox{ and } \\ 
I''' &= I \setminus (I' \cup I''). \\ 
\end{align*}
Then we have a decomposition of $I=I' \: \dot{\cup}\: I'' \: \dot{\cup} \:  I'''$ into disjoint sets, and the tensor product of $L_n$ and $V_{\lambda}$ decomposes as: 
\begin{align} 
L_n\otimes V_\lambda\cong \bigoplus_{r\in I'} T_{\lambda +r}\oplus\bigoplus_{s\in I'''} V_{\lambda+s}.  \label{tensorproductdecomposition}
\end{align}  
\end{proposition}   

In the setting of our paper, we begin by choosing the case $\lambda=0$ so that we have the decomposition 
\begin{align} 
L_n\otimes V_0\cong \bigoplus_{r\in I'} T_{r}\oplus\bigoplus_{s\in I'''} V_{s},   \label{tensorproductdecomposition}
\end{align}  
where 
\begin{align*}
I'&=\begin{cases}
\{1,3,5,\dots, n-4,n-2\} &\text{ if $n$ is odd}, \\
\{0,2,4,\dots, n-4,n-2\} &\text{ if $n$ is even}, \\
\end{cases} \\ 
I''&=\begin{cases} 
\{-n,-n+2,\dots,-3\} &\text{ if $n$ is odd}, \\
\{-n,-n+2,\dots, -2\} &\text{ if $n$ is even},
\end{cases} \\ 
I'''&=\begin{cases}
\{-1,n\} &\text{ if $n$ is odd},  \\
\{n\} &\text{ if $n$ is even}. 
\end{cases}
\end{align*}


\begin{definition} 
Let $c\in \mathbb{C}$ and let $M$ be an $\mathfrak{sl}_2$-module. Define 
\[
M[c]:= \{ a\in M: (\Omega-c)^i.a = 0 \mbox{ for some }i\in \mathbb{N}_{>0}\},  
\] 
the {\em $c$-generalized eigenspace} of the Casimir operator in $M$. We say $M$ has a 
{\em Casimir decomposition} if $M=\bigoplus_{c\in \mathbb{C}}M[c]$ as modules, and that $M$ has a 
{\em simple Casimir decomposition} if $(\Omega-c).M[c]=0$ for every $c\in \mathbb{C}$.   
\end{definition} 

Theorem~\ref{theorem:enright-casimir-decomposition} is proved in \cite{MR541329}: 
\begin{theorem}\label{theorem:enright-casimir-decomposition}
Let $M$ be an $\mathfrak{sl}_2$-weight module, where $e\in \mathfrak{sl}_2$ acts nilpotently. 
Then $M$ has a Casimir decomposition.  
\end{theorem}

\begin{lemma} 
If an $\mathfrak{sl}_2$-module $M$ has a Casimir decomposition, then $(\Omega-c)^2.M[c]=0$. 
\end{lemma} 

Enright also gives the following definition:   
\begin{definition}\label{def:enright-category-Ig-sl2-c} 
For a complex Lie algebra $\mathfrak g$ containing $\mathfrak{sl}_2$, the category $\mathcal I_{\mathfrak g}(\mathfrak{sl}_2)[c]$ is the collection of all $\mathfrak{sl}_2$-modules $M$ in $\mathcal I_{\mathfrak g}(\mathfrak{sl}_2)$ such that $M=M[c]$.  
\end{definition}

Definition~\ref{def:enright-category-Ig-sl2-c} implies 
for $n\geq -1$, 
the modules $V_{-n-2}$, $V_n$ and $T_n$ are in the category $\mathcal I_{\mathfrak g}(\mathfrak{sl}_2)[n(n+2)]$.

We now refer to \cite{Benes} for the decomposition of tensor product of modules. 
Using the above bases for $L_n$ and $V_0$,  
we will construct a basis for $T_r$ and $V_s$. 
That is, we know that $\bigoplus_{r\in I'} T_r \oplus \bigoplus_{s\in I'''} V_s \cong L_n\otimes V_0$ has basis 
$\{v_i\otimes w_k: 0\leq i\leq n, k\geq 0 \}$. 
After constructing a basis for $T_r$ and $V_s$, we replace the basis $\{v_i \otimes w_k\}$ 
with the tensor product $M_n\boxtimes N_k$ of modules, 
where $M_n$ is an object in category $\mathcal L_n$ and $N_k$ is an object in $\mathcal V_0$.

The Casimir $\Omega=4fe+h^2+2h$ acts on $v_{\lambda}\in V_{\lambda}$ by 
$\Omega. v_{\lambda} = (4fe+h^2+2h)v_{\lambda}=(\lambda^2+2\lambda)v_{\lambda}=\lambda(\lambda+2)v_{\lambda}$. 
Thus $(\Omega-\lambda(\lambda+2))v_{\lambda}=0$. 
Now, for a fixed $s$ and a fixed $n$, 
we will find the highest weight vector $u_s$ of $V_s$ with weight $s$
in terms of a sum of tensor products of vectors in $ L_n\otimes V_0$.   If the coefficients of the sum are positive integers, one can lift this sum to a sum of tensor products of modules $M_n\boxtimes N_k\in \mathcal L_n\boxtimes  \mathcal{V}_0$.  
\begin{proposition}\label{proposition:p-j-for-Verma-mod-cat}  
For $s\in I'$ set $r=-s-2$. For $0\leq i\leq \frac{n+r}{2}$, set
\begin{equation}
p_i :=4^{\frac{n+r}{2}-i} \dfrac{n+r+2}{n+r-2i+2} \prod_{j=0}^{i-1} (n+r-2j)^2 \prod_{\nu=i}^{\frac{n+r-2}{2}} (n-\nu), 
\end{equation}
where if $i=0$, then 
\[ 
\prod_{j=0}^{i-1} (n+r-2j)^2: = 1. 
\]  
Then $p_i$ is a positive integer, and  
$u_n:=v_0\otimes w_0$ is a highest weight vector in $L_n\otimes V_0$ of highest weight $n$ and 
 \begin{equation}
u_s
:= \sum_{j=0}^{\frac{n+r}{2}}p_j v_j \otimes w_{\frac{n+r+2}{2}-j },  
\end{equation}
a highest weight vector of weight $s$. 
Moreover  
 \begin{equation}
u_{-s-2}=f^{s+1}u_s
 := 
\sum_{j=0}^{\frac{n+s}{2} } p_jv_j \otimes w_{\frac{n+s+2}{2} - j }, 
\end{equation}
a highest weight vector of weight $-s-2$.
\end{proposition}
\begin{proof}  When $s=n$, where $s\in I'''$, it is clear that $u_n=u_s=v_0\otimes w_0$ is a highest weight vector of weight $s$. Now let $s \in I'$.   By the Decomposition Theorem \eqnref{tensorproductdecomposition} and the structure of the module $T_s$, there is only one highest weight vector of weight $s$ up to a scalar. 
As  
$$
h.(v_i\otimes w_k)
= (n-2i-2k)v_i\otimes w_k,
$$
the highest weight vector $u_s$ of $V_s$ is a linear combination of $v_i\otimes w_k$ satisfying $n-2i-2k=s$. 
We will determine coefficients  $ \alpha_{ik}\in \mathbb C$ using the equation 
\[ 
u_s = \sum_{\stackrel{0\leq i\leq n,k\geq 0}{n-2i-2k=s}} \alpha_{ik} v_i\otimes w_k, 
\mbox{ where } V_s = \mathbb{C}[f]u_s. 
\] 
Since 
\[
\begin{aligned}
0 &= eu_s = \sum_{\stackrel{0\leq i\leq n, k\geq 0}{n-2i-2k=s}} \alpha_{ik} e(v_i\otimes w_k)  
= \sum_{\stackrel{0\leq i\leq n, k\geq 0}{n-2i-2k=s}} \alpha_{ik} (n-i+1)v_{i-1}\otimes w_k  \\ 
&\hspace{2cm}+ \sum_{\stackrel{0\leq i\leq n, k\geq 0}{n-2i-2k=s}} \alpha_{ik}(-k(k-1))v_i\otimes w_{k-1} \\ 
&= \sum_{\stackrel{-1\leq i'\leq n-1, k\geq 0}{n-2i'-2k-2=s}} \alpha_{i'+1,k} (n-i')v_{i'}\otimes w_k 
+   \sum_{\stackrel{0\leq i\leq n, k'\geq -1}{n-2i-2k'-2=s}} \alpha_{i,k'+1}(-(k'+1)k')v_i\otimes w_{k'}  \\ 
&\hspace{2cm}\mbox{ since } i-1=i',  k-1 = k'   \\ 
&= \sum_{\stackrel{0\leq i'\leq n-1, k\geq 0}{n-2i'-2k-2=s}} \alpha_{i'+1,k} (n-i')v_{i'}\otimes w_k 
+   \sum_{\stackrel{0\leq i\leq n, k'\geq 0}{n-2i-2k'-2=s}} \alpha_{i,k'+1}
\big(-(k'+1)k' \big)v_i\otimes w_{k'}  \\ 
&\hspace{2cm}\mbox{ since } v_{-1} =  w_{-1} = 0 \\ 
&= \sum_{\stackrel{0\leq i\leq n-1, k\geq 0}{n-2i-2k-2=s}} 
\big(\alpha_{i+1,k} (n-i)  - \alpha_{i,k+1}(k+1)k \big)  v_i\otimes w_{k}  
+ \alpha_{n,k+1} \big(-(k+1)k \big) v_n\otimes w_{k},   
\end{aligned}
\] 
we have $\alpha_{n,k+1}=0$ for all $k\geq 0$ and 
$\alpha_{i+1,k} (n-i)  - \alpha_{i,k+1}(k+1)k=0$ for each $0\leq i\leq n-1$, $k\geq 0$, and 
$n-2i-2k-2=s$. 
The latter equation could be rewritten as 
\begin{equation}\label{equation:relation-alpha-ij}
\alpha_{i+1,k}  = \dfrac{(k+1)k}{n-i} \alpha_{i,k+1} \mbox{ where } n-2i-2k-2=s,  0\leq i\leq n-1, \mbox{ and } k\geq 0.   
\end{equation}

Furthermore, since $k = (n-2i-s-2)/2$, we can rewrite Equation~\eqref{equation:relation-alpha-ij} as:  
\begin{align*}
\alpha_{i+1,(n-s  -2i-2)/2} &= \dfrac{(n-s -2i)(n-s  -2i-2)}{4(n-i)} \alpha_{i,(n-s  -2i)/2} \\  &= \dfrac{n-i-1}{4^{i+1}(n-s)(n-s-2i-2)} \prod_{j=0}^{i+1} \dfrac{(n-s-2j)^2}{n-j}   \alpha_{0,(n-s )/2}, \\ 
\end{align*}
where $0\leq i\leq n-1$, 
or as: 
\begin{align*}
\alpha_{i,(n-s  -2i)/2} &= \dfrac{(n-s -2i+2)(n-s  -2i)}{4(n-i+1)} \alpha_{i-1,(n-s  -2i+2)/2}   \\ 
&= \dfrac{(n-s-2i)}{4^i(n-s)} \prod_{j=0}^{i-1}\dfrac{(n-s-2j)^2}{n-j}\alpha_{0,(n-s)/2},  \\ 
\end{align*}
where $1\leq i\leq n$. 
If we set $r=-s-2$, 
then for $1\leq i\leq n$, we get 
\begin{align*}
\alpha_{i,(n+r  -2i+2)/2} 
&= \dfrac{(n+r+2)}{4^i(n+r-2i+2)} \prod_{j=0}^{i-1} 
\dfrac{(n+r-2j)^2 }{(n-j)} \alpha_{0,(n+r+2 )/2}.   \\
\end{align*}  
Note that since $\frac{n+r-2i}{2}=k\geq 0$, we have $\frac{n+r}{2}\geq i$. 
As a consequence, we assume $\alpha_{0,(n+r+2 )/2}  \neq 0$ in order to have all the coefficients $\alpha_{i,(n+r  -2i+2)/2}\neq 0$ for  
$i\leq \frac{n+r}{2}$,    
%
and we also assume $\alpha_{i,(n+r  -2i+2)/2}=0$ for $i\geq \frac{n+r+2}{2}$.  
Moreover if we set the initial term as 
$$
\alpha_{0,(n+r+2 )/2} =4^{\frac{n+r}{2}} \prod_{\nu=\frac{n-r+2}{2}}^n \nu, 
%
$$
then $\alpha_{i,(n+r  -2i+2)/2} $ will be positive integers for all $0 \leq i\leq \frac{n+r}{2}$. 
Hence we may take the highest weight vector $u_s$ to be  
\begin{equation*}
u_s    
= \sum_{j=0}^{\frac{n+r}{2}} p_j v_j \otimes w_{\frac{n+r+2}{2}-j}, 
\end{equation*}
where $p_i$ is the positive integer   
given by 
\begin{equation*}
p_i=4^{\frac{n+r}{2}-i} \dfrac{n+r+2}{n+r-2i+2} \prod_{j=0}^{i-1} (n+r-2j)^2 \prod_{\nu=i}^{\frac{n+r-2}{2}} (n-\nu),  
\end{equation*}
where $0 \leq i\leq \frac{n+r}{2}$ and $0\leq s=-r-2 \leq n-2$,  
and if $i=0$, we define 
\[ 
\prod_{j=0}^{i-1} (n+r-2j)^2 := 1. 
\] 
 Now observe that if 
 \begin{equation*}
 w=\sum_{\stackrel{0\leq i\leq n,k\geq 0}{n-2i-2k=s}} \gamma_{ik} v_i\otimes w_k, 
 \end{equation*}
a weight vector with nonnegative coefficients $\gamma_{ik}\in\mathbb N$, then setting $v_{-1}=0$ and $w_{-1}=0$, we get
 \begin{align*}
 fw&=\sum_{\stackrel{0\leq i\leq n,k\geq 0}{n-2i-2k=s}} \gamma_{ik} (fv_i\otimes w_k+v_i\otimes fw_k)  \\
 &=\sum_{\stackrel{0\leq i\leq n,k\geq 0}{n-2i-2k=s}} \gamma_{ik} (i+1)v_{i+1}\otimes w_k+\sum_{\stackrel{0\leq i\leq n,k\geq 0}{n-2i-2k=s}}\gamma_{ik} v_i\otimes w_{k+1}  \\
 &=\sum_{\stackrel{0\leq i\leq n,k\geq 0}{n-2i-2k+2=s}} (i \gamma_{i-1,k}  + \gamma_{i,k-1} )v_i\otimes w_{k}, 
 \end{align*}
which also has coefficients $ i \gamma_{i-1,k}  + \gamma_{i,k-1}$ as non-negative integers.  Hence by induction, every vector $f^lu_s$ has non-negative coefficients in front of the summands $v_i\otimes w_{k}$.
In particular,  
 \begin{equation}
u_{-s-2}=f^{s+1}u_s  :=\sum_{j=0}^{\frac{n+s}{2}} q_jv_j \otimes w_{\frac{n+s+2}{2} - j}, 
\end{equation}
where $q_j$ are non-negative integers and $n-2i-2k-2=-s-2$.  The fact that the coefficients above are nonnegative integers tells us that we are dealing with an additive category.

Now for the comultiplication 
$\Delta: U(\mathfrak{sl}_2) \to U(\mathfrak{sl}_2) \otimes U(\mathfrak{sl}_2)$ given by $\Delta(x)=x\otimes 1+1\otimes x$, we get 
$$
\Delta(f^{s+1})=\Delta(f)^{s+1}=(f\otimes 1+1\otimes f)^{s+1}=\sum_{k=0}^{s+1}\binom{s+1}{k}f^k\otimes f^{s+1-k}
$$
and thus 
 \begin{align*}
u_{-s-2}& 
= f^{s+1}u_s=\sum_{k=0}^{s+1}\binom{s+1}{k}f^k\otimes f^{s+1-k}
	\left(  
		\sum_{j=0}^{\frac{n+r}{2}} p_j v_j \otimes w_{\frac{n+r+2}{2} -j} 
	\right) \\
&= p_0 v_0\otimes  f^{s+1} w_{\frac{n+r+2}{2}}+ q_1 v_1\otimes w_{\frac{n+s}{2}}+\cdots +q_{\frac{n+s}{2}}v_{\frac{n+s}{2}}\otimes w_{1} \\ 
&= p_0 v_0\otimes  w_{\frac{n+s+2}{2}}+ q_1 v_1\otimes w_{\frac{n+s}{2}}+\cdots +q_{\frac{n+s}{2}}v_{\frac{n+s}{2}}\otimes w_{1}. 
\end{align*}
Given the commutator $[e,f^{s+1}]=(s+1)f^s(h-s)$, one of course has $eu_{-s-2}=0$.  This implies 
 \begin{align*}
e u_{-s-2} &= p_0v_0\otimes ew_{\frac{n+s+2}{2}}+\cdots +q_i(ev_i\otimes w_{\frac{n+s-2i+2}{2}}+v_i\otimes ew_{\frac{n+s-2i+2}{2}}) \\
&\quad\quad +\cdots +q_{\frac{n+s}{2}}ev_{\frac{n+s}{2}}\otimes w_{1} \\
&= p_0v_0\otimes w_{\frac{n+s+2}{2}}+\cdots +q_i(n-i+1)v_{i-1}\otimes w_{\frac{n+s-2i+2}{2}} \\ 
&\quad\quad - q_i\left(\frac{n+s-2i+2}{2}\right)\left(\frac{n+s-2i}{2}\right)v_i\otimes w_{\frac{n+s-2i}{2}} 
+ q_{i+1}(n-i)v_{i}\otimes w_{\frac{n+s-2i}{2}} \\ 
&\quad\quad - q_{i+1}\left(\frac{n+s-2i}{2}\right)\left(\frac{n+s-2i-2}{2}\right)v_{i+1}\otimes w_{\frac{n+s-2i-2}{2}} \\
&\quad\quad +\cdots +q_{\frac{n+s}{2}}\left(\frac{n-s+2}{2}\right)v_{\frac{n+s-2}{2}}\otimes w_{1}, 
\end{align*}
which gives us 
\begin{align*}
q_{i+1}(n-i)=q_i\left(\frac{n+s-2i+2}{2}\right)\left(\frac{n+s-2i}{2}\right)=q_i(k+1)k
\end{align*}
since $k=\frac{n+s-2i}{2}$.   Since $q_0=p_0$, we have $q_i=p_i$ for all $i$.

\end{proof}

%
%
\begin{proposition}  
For $s\in I'$, 
there is a vector in the tensor product $L_n\otimes V_0$ of the form
 \begin{equation}
a_{-s-2}:= \sum_{j=0}^{\frac{n+s}{2}} q_j v_j\otimes w_{\frac{n+s+2}{2} - j}  
\end{equation}
of weight $-s-2$ 
which generates the projective submodule $T_s$, where for $0\leq i\leq (n+s)/2$, $q_i$'s are positive integers satisfying the relation 
\begin{equation}\label{equation:recursion-equation-basis-coeff-betas}
\begin{split}
0 &=q_{i-2 } \big(i-1\big) i k \big(k+1\big)^2 \big(k+2 \big)    
	-  q_{i-1 } \Big( 
         i k (k+1)  ( 2 i( n +2 - i) - (n + 2)  -2k^2 )         
	\Big) \\ 
	&\hspace{4mm}	 +  q_{i} \Big(  
			(i(n-i+1) - k(k-1))^2     
		  -   ik (k+1) (n-i+1)        
		  -   (i+1) (k-1) k  (n-i)  \Big)  \\ 
	&\hspace{4mm}	  +  q_{i+1 } (n-i) \Big( n+2i(n-i) -2(k-1)^2  \Big)       
	 	  + q_{i+2 } \big(n-i-1 \big) \big(n-i \big),  
\end{split}
\end{equation}
where $k=(n-2i+s+2)/2$.
\end{proposition}
\begin{proof}
We want to find a weight vector 
$a_{-s-2}$ 
of weight $-s-2$ such that 
$(\Omega-c)^2. a_{-s-2}=0$ but $(\Omega -c). a_{-s-2}\not=0$, where $c=s(s+2)$, since such an $a_{-s-2}$ would be a generator of $T_s$ as a $U(\mathfrak{sl}_2)$-module.  
Similar to the calculation for the highest weight vector of a Verma module, we have 
%
\[  
\begin{aligned}  
\Omega.(v_i\otimes w_k) &= (4fe+h^2+2h).(v_i\otimes w_k)   \\
&=  - 4k(k-1)(i+1)v_{i+1}\otimes w_{k-1}
+4(n-i+1) v_{i-1} \otimes w_{k+1} \\ 
&\quad   
+ \big(4(n-i+1)i + (n - 2i - 4k + 2)(n-2i) \big)v_i\otimes w_k.  \\
\end{aligned} 
\]  

Since $c_s = s(s+2)=(n-2i-2k)(n-2i-2k+2)= (n-2i)^2 - 4k(n-2i) + 4k^2 + 2(n-2i-2k)$ and 
$c_s = c_{-s-2}$, 
we obtain 
\[ 
\begin{aligned}
(\Omega - c_s )(v_i\otimes w_k )
&= 4 (i (n-i+1) - k(k-1)) v_i\otimes w_k - 4k(k-1)(i+1)v_{i+1}\otimes w_{k-1}  \\ 
 &\quad\quad + 4(n-i+1)v_{i-1}\otimes w_{k+1}.  \\ 
\end{aligned}
\] 

Now as 
\[ 
a_{-s-2} = \sum_{\stackrel{0\leq i\leq n, k\geq 0}{n-2i-2k=-s-2} } \beta_{ik}(v_i\otimes w_k)   
\]  
is a weight vector of weight $-s-2$ in the projective module $T_{-s-2}$, 
we want to show that there exist coefficients $\beta_{ik}$ that are integers and such that $a_{-s-2}$ is a generator of $T_{-s-2}$.  So as mentioned above we want $(\Omega - c_s).a_{-s-2}\neq 0$ but $(\Omega - c_s)^2.a_{-s-2}=0$.  First we have
\[  
\begin{aligned}  
(\Omega &- c_s).a_{-s-2}  \\  
	&=   \sum_{\stackrel{0\leq i\leq n, k\geq 0}{n-2i-2k=-s-2}} \beta_{ik} 
	\Big(4i(n-i+1)-4k(k-1) \Big) v_i\otimes w_k   \\ 
	 &\hspace{4mm} -   \sum_{\stackrel{0\leq i\leq n-1, k\geq 1}{n-2i-2k=-s-2}}\beta_{ik}(4k)(k-1)(i+1) v_{i+1}\otimes w_{k-1}   
	 + \sum_{\stackrel{1\leq i\leq n, k\geq 0}{n-2i-2k=-s-2}} \beta_{ik} 4 \big(n-i+1 \big) v_{i-1}\otimes w_{k+1}  \\
	 & \hspace{4mm} + \sum_{\stackrel{0\leq i\leq n, k\geq 0}{n-2i-2k=-s-2}} \beta_{ik} 
	\Big(4i(n-i+1)-4k(k-1) \Big) v_i\otimes w_k    \\ 
	&= 
	 -   \sum_{\stackrel{1\leq i'\leq n, k'\geq 0}{n-2i'-2k' = -s-2}} \beta_{i'-1, k'+1}\big(4(k'+1) \big)(k')(i') v_{i'}\otimes w_{k'} +   \sum_{\stackrel{0\leq i''\leq n-1, k''\geq 1}{n-2i'' -2k'' = -s-2}} \beta_{i''+1, k''-1} 4 \big( n-i'' \big) v_{i''}\otimes w_{k''}     \\ 
	&=  \sum_{\stackrel{0\leq i\leq n, k\geq 0}{n-2i-2k=-s-2}} \beta_{ik} 
	 \Big( 4i(n-i+1)-4k(k-1) \Big)v_i\otimes w_k  
	 -   \sum_{\stackrel{0\leq i\leq n, k\geq 0}{n-2i-2k = -s-2}} \beta_{i-1, k+1} 4ki 
	  (k+1) v_{i}\otimes w_{k}  \\ 
	&\hspace{4mm} +   \sum_{\stackrel{0\leq i\leq n, k\geq 0}{n-2i -2k = -s-2}} \beta_{i+1, k-1} 4\big(n-i \big) v_{i}\otimes w_{k} 
		\hspace{4mm} \mbox{ since } \beta_{-1,k+1} :=0 \mbox{ and } \beta_{i+1,-1} := 0  \\   
	&= \sum_{\stackrel{0\leq i\leq n, k\geq 0}{n-2i-2k=-s-2}} 
		\Big( 
			\beta_{ik} \big(4i(n-i+1)-4k(k-1) \big)   -  \beta_{i-1, k+1} 4ki (k+1)   +  \beta_{i+1, k-1} 4(n-i) \Big)  v_{i}\otimes w_{k},   \\  
\end{aligned}  
\]  
where $ i'=i+1$, $k' = k-1$, $i'' = i-1$, and $k'' = k+1$ hold in the second equality. 

Let $\zeta_{ik}:= \beta_{ik} \big(4i(n-i+1)-4k(k-1) \big)   -  \beta_{i-1, k+1} 4ki 
\big(k+1 \big)   +  \beta_{i+1, k-1} 4 \big( n - i \big) $. 
Then  
\[ 
\begin{aligned}  
0 &= (\Omega - c_s)^2 .a_{-s-2} \\ 
&= 	\sum_{\stackrel{0\leq i\leq n, k\geq 0}{n-2i-2k=-s-2}} 
	\Big( \zeta_{ik} (4i(n-i+1)-4k(k-1))   -  \zeta_{i-1, k+1} 4ki (k+1)   +  \zeta_{i+1, k-1} 4(n-i)  \Big)  v_{i}\otimes w_{k}   \\  
&= 	\sum_{\stackrel{0\leq i\leq n, k\geq 0}{n-2i-2k=-s-2}} 
	\Big(  \Big(\beta_{ik} \big(4i(n-i+1)-4k(k-1) \big)   -  \beta_{i-1, k+1} 4ki (k+1)   +  \beta_{i+1, k-1} 4(n-i) \Big) \cdot  \\ 
	&\hspace{8mm}\cdot  \Big(4i (n-i+1)-4k(k-1) \Big)    
	- \Big(   \beta_{i-1,k+1} \Big( 4(i-1)(n-i+2) - 4k(k+1) \Big)   \\ 
	&\hspace{6mm} - \beta_{i-2, k+2} 4(k+1)(i-1) (k+2)   
	+  \beta_{i, k} 4(n-i+1)  \Big)     4ki \big( k+1 \big)   \\ 
	&\hspace{6mm} +  \Big( \beta_{i+1,k-1} \Big( 4(i+1)(n-i)-4(k-1)(k-2) \Big)  \\ 
	& \hspace{6mm}   -  \beta_{i, k} 4 k (k-1)(i+1)   +  \beta_{i+2, k-2} 4(n-i-1) \Big)  \cdot 
 		4(n-i) \Big)  \cdot   v_{i}\otimes w_{k}   \\   
&= 16	\sum_{\stackrel{0\leq i\leq n, k\geq 0}{n-2i-2k=-s-2}} 
	 \Big(   \beta_{i-2, k+2} (i-1)i k (k+1)^2 (k+2)     \\ 
&\hspace{6mm}	  - \beta_{i-1, k+1} \Big(  
         i k (k+1)  \big( 2 i( n +2 - i) - (n + 2)  -2k^2 \big)         
	\Big) \\ 
	&\hspace{6mm}	 +  \beta_{ik} \Big(  
			\Big( i (n-i+1) - k(k-1) \Big)^2     
		  -   ik (k+1) (n-i+1)        
		  -   (i+1) (k-1) k  (n-i)  \Big)  \\ 
	&\hspace{6mm}	\left.     +  \beta_{i+1, k-1} \big( n - i \big) \Big(n+2i(n-i) -2(k-1)^2  \Big)      
	+  \beta_{i+2, k-2} (n-i-1)(n-i)  \right)v_{i}\otimes w_{k}.  
\end{aligned}  
\]  
This means the coefficient of each $v_i\otimes w_k$ must equal zero: 
\begin{equation}\label{eq:deriving-beta-ik-for-s-2} 
\begin{split}
0 &= \beta_{i-2, k+2} (i-1)i k (k+1)^2 (k+2)    
	-  \beta_{i-1, k+1} \Big( 
         i k (k+1)  \big( 2 i( n +2 - i) - (n + 2)  -2k^2 \big)         
	\Big) \\ 
	&\hspace{4mm}	 +  \beta_{ik} \Big(  
			\big( i (n-i+1) - k(k-1) \big)^2     
		  -   ik (k+1) (n-i+1)        
		  -   (i+1) (k-1) k  (n-i)  \Big)  \\ 
	&\hspace{4mm}	  +  \beta_{i+1, k-1} (n-i)\Big( n+2i(n-i) -2(k-1)^2  \Big)       
		  +  \beta_{i+2, k-2} \big( n-i-1 \big) \big( n - i \big) 
\end{split}
\end{equation}
for each $0\leq i\leq n$, $k\geq 0$, and $n-2i-2k=-s-2$. 
Note that $n+s$ must be even since $n+s = 2(i + k -1)$.  

When $k=0$, we have $i = \frac{n+s+2}{2}$ and 
$(i(n-i+1)  )^2 \beta_{i0} = 0. $ 
Since $s\in I'$, $0\leq s\leq n-2$ if $n$ is even or $1\leq s\leq n-2$ if $n$ is odd. 
So $i\not=0$ and $i\leq n$. Thus, $\beta_{i0}=0$ for all $i$.

Now, provided $2\leq  i< \frac{n+s+2 }{2}$ (with $n-2i-2k=-s-2$), we have
\begin{align}
&\beta_{i-2, \frac{n+s+6 -2i}{2} }   \label{betarecursion}  \\ 
   &= \frac {16}{   i(i-1) (n + s + 6 - 2i )(  n + s + 4 - 2i  )^2  (n+s+2-2i )    } \Bigg(  \beta_{i-1, \frac{n + s + 4 - 2i}{2}  }\cdot \notag\\
   &\quad \cdot  \left(  
          \dfrac{ i (n+s+2-2i) (  n+s+4-2i     )}{4}       \left( 2 i(n - i +2 ) - (n + 2)  -   2   \left( \dfrac{ n+s+2-2i }{2 }  \right)^2 
\right)    
	\right)  \notag\\ 
	 	 &\qquad-  \beta_{i, \frac{n+s+2-2i}{2}  } \left(  
			\left(i(n-i+1) -    \dfrac{(n+s+2-2i  ) ( n+s-2i)}{4}        \right)^2  \right.  \notag  \\ 
		  & \qquad  \left.   
  -  \dfrac{ n+s+2-2i }{4}\left(    i (n-i+1)   (n+s+4-2i)   
		  +     (n+s -2i) (i+1)(n-i)    
\right) 
				\right)       \notag \\   
	 	  &\qquad-  \beta_{i+1, \frac{n+s-2i}{2}  } (n-i) \left(      n  + 2 i (n-i)   -   \dfrac{( n + s - 2i )^2}{2}   \right)   
 	         -  \beta_{i+2, \frac{n+s-2-2i}{2} } (n-i-1)(n-i)\Bigg).   \notag  
\end{align}

For $k=1$, $i=\frac{n+s}{2}$ and the recursion in \eqref{eq:deriving-beta-ik-for-s-2} gives  
\begin{align*}
0 &= 3 \big( n+s-2 \big) \big( n+s \big) \beta_{\frac{n+s-4}{2}, 3 }    
  - \frac{n+s}{2}  \Big(n^2+2 n-s^2+4 s-8\Big) \beta_{\frac{n+s-2}{2}, 2 }  \\ 
   &\hspace{4mm}+  \frac{(n-s+2) (n+s)}{16}  \Big(n^2+2 n-s^2+2 s-8\Big)  \beta_{\frac{n+s}{2}, 1 }.   \\   
\end{align*}   
Recall the recurrence relation in Equation~\eqref{equation:relation-alpha-ij}. 
Setting 
\begin{align*}
 \beta_{\frac{n+s-4}{2}, 3 } & =  \alpha_{\frac{n+s-4}{2}, 3 }, \\
 \beta_{\frac{n+s-2}{2}, 2 } & = \alpha_{\frac{n+s-2}{2}, 2 }=  \dfrac{12}{(n-s+4)} \alpha_{\frac{n+s-4}{2},3},  \\
   \beta_{\frac{n+s}{2}, 1 }& =   \alpha_{\frac{n+s}{2}, 1 }= \dfrac{4}{(n-s+2)} \alpha_{\frac{n+s-2}{2},2} 
= \dfrac{48}{(n-s+2)(n-s+4)} \alpha_{\frac{n+s-4}{2},3}, 
\end{align*}
we obtain the relation 
\begin{align*}
& (n+s-2)   -      \dfrac{2 \left(n^2+2 n-s^2+4 s-8\right)}{(n-s+4)}  + \dfrac{  \left(n^2+2 n - s^2 + 2 s - 8 \right)}{ (n-s+4)}   
		 =0\end{align*} 
as one should for the expansion of the highest weight vector $u_{-s-2}$.

Now pick $\beta_{i,k}$ as rational numbers so as to satisfy \eqnref{betarecursion}  with $ \beta_{\frac{n+s}{2}, 1 }=   \alpha_{\frac{n+s}{2}, 1 }$ but  $\beta_{\frac{n+s-2}{2}, 2 } \neq  \alpha_{\frac{n+s-2}{2}, 2 }$. This will give us a generator $a_{-s-2}$ with rational coefficients that satisfies 
$$
 (\Omega-c_{-s-2})^2.a_{-s-2} =0 \mbox{ and }  (\Omega-c_{-s-2}).a_{-s-2}\neq 0.
 $$
After clearing denominators, we may assume $\beta_{i,k}\in\mathbb Z$.

In conclusion, we have
\begin{equation}
u_{-s-2}= \sum_{j=0}^{\frac{n+s}{2}} p_j v_j \otimes w_{\frac{n+s+2}{2}-j }
\end{equation}
with $p_j>0$ integers and 
\begin{equation}
a_{-s-2}= \sum_{j=0}^{\frac{n+s}{2}}  q_j v_j \otimes w_{\frac{n+s+2}{2} -j }
\end{equation}
with $q_j\in\mathbb Z$. Now there exists $m>0$ such that 
\begin{equation}
a_{-s-2}+mu_{-s-2} = 
\sum_{j=0}^{\frac{n+s}{2}} (q_j + mp_j)v_j \otimes w_{\frac{n+s+2}{2}-j }  
\end{equation}
 and $q_j+mp_j>0$ for all $0\leq j\leq \frac{n+s}{2}$.  Thus we can replace $a_{-s-2}$ with $a_{-s-2}+mu_{-s-2}$ as the following conditions:  
 $$
 (\Omega-c_{-s-2})^2. (a_{-s-2}+mu_{-s-2})=0   \mbox{ and } 
 (\Omega-c_{-s-2}). (a_{-s-2}+mu_{-s-2}) = (\Omega-c_{-s-2}). a_{-s-2}\neq 0 
 $$
are still satisfied. 
 \end{proof}

Since the ``new" vector $a_{-s-2}+mu_{-s-2}$ will have positive integer coefficients, this vector will correspond to the equivalence class of a sum of  exterior tensor products of modules of the form 
 $$
 (M_0\boxtimes N_{\frac{n+s+2}{2}})^{\oplus (q_0+mp_0)}\oplus (M_1\boxtimes N_{\frac{n+s}{2}})^{\oplus (q_1+mp_1)}\oplus \cdots \oplus 
 (M_{\frac{n+s}{2}}\boxtimes  N_{1})^{\oplus (q_{\frac{n+s}{2}} + mp_{\frac{n+s}{2}})}  
$$
in $\mathcal{L}_n\boxtimes  \mathcal{V}_0$.

\begin{corollary}\label{corollary:basis-Ts}
 The set $\{f^ka_{-r-2} :  k\in\mathbb{Z}_{\geq 0}\}\cup\{f^ku_r :  k \in \mathbb{Z}_{\geq 0}\}$ is a basis of the projective module $T_r$.  Consequently, $T_r$ has a basis consisting of nonnegative sums of vectors of the form $v_i\otimes w_l$. 
\end{corollary}
\begin{proof} It follows from \cite{MR541329} that the above set is a basis.  The fact that the coefficients can be taken to be positive integers is what is proven above.
\end{proof}

\begin{remark}  If one considers the Clebsch-Gordan decomposition of $L_1\otimes L_1\cong L_{2}\oplus L_{0}$ with respect to the basis $v_i\otimes v_j$, we have a highest weight vector of weight $0$ of the form
\begin{equation}
u_0=v_0\otimes v_1-v_1\otimes v_0.
\end{equation}
In particular, the coefficients of the tensors $v_i\otimes v_j$ in the decomposition of a highest weight vector inside of $L_m\otimes L_n$ can be negative.   This is in contrast to Corollary~\ref{corollary:basis-Ts}. 
 \end{remark}

\subsection{Affine Hecke algebras}  In this section, we follow the notation in \cite{MR2373155}.   Assume throughout that $\mathbb k$ is a field and $q\in\mathbb k^*$. 
\subsubsection{Nondegenerate affine Hecke algebras}\label{subsubsection:nondegenerate-affine-Hecke-algebras}   Here we take $q\neq 1$.  Recall the affine Hecke algebra $H_n = H_n(q)$ is the $\mathbb k$-algebra with generators
\[
T_1,\dots, T_{n-1},X_1^{\pm 1},\dots, X_n^{\pm 1}
\]
with defining relations
\begin{align}
\mbox{eigenvalue relations:}\hspace{2cm}& \\ 
(T_i+1)(T_i-q)&=0, \\
\mbox{braid relations:}\hspace{2cm}& \\ 
T_iT_j&=T_jT_i \quad \text{ if }|i-j|>1, \\ 
T_iT_{i+1}T_i&=T_{i+1}T_iT_{i+1}, \\ 
\mbox{Laurent relations:}\hspace{2cm}& \\ 
X_iX_i^{-1}&=X_i^{-1}X_i=1,\\
X_iX_j&=X_jX_i,  \\
\mbox{action relations:}\hspace{2cm}& \\ 
X_i T_j &= T_j X_i \quad \text{ if }i-j\neq 0,1,  \\ 
T_iX_iT_i&=qX_{i+1}. 
\end{align}
The subalgebra of $H_n(q)$ generated by $T_1,\dots, T_{n-1}$ is denoted by $H_n^A(q)$ which is isomorphic to the Hecke algebra of the symmetric group $\mathfrak S_n$.   The subalgebra of Laurent polynomials in $X_1^{\pm 1},\dots, X_n^{\pm 1}$ is denoted by $P_n=\mathbb k[X_1^{\pm 1},\dots ,X_n^{\pm1}]$.

\subsubsection{Degenerate affine Hecke algebras}\label{subsubsection:degenerate-affine-Hecke}  
Here one sets $q=1$.   Recall the degenerate affine Hecke algebra $H_n(1)$ is the $\mathbb k$-algebra with generators
\[
T_1,\dots, T_{n-1},X_1,\dots, X_n
\]
with defining relations
\begin{align}
T_i^2&=1, \\
T_iT_j&=T_jT_i \quad \text{ if }|i-j|>1, \\ 
T_iT_{i+1}T_i&=T_{i+1}T_iT_{i+1}, \\ 
X_i X_j &= X_j X_i,  \\ 
X_iT_j&=T_jX_i \quad \text{ if }i-j\neq 0,1, \\
X_{i+1}T_i&=T_iX_{i}+1.
\end{align}
The subalgebra of $H_n(1)$ generated by $T_1,\dots, T_{n-1}$ is denoted by $H_n^A(1)$ which is isomorphic to the group algebra of the symmetric group $\mathfrak S_n$.   The subalgebra of polynomials in $X_1 ,\dots, X_n$ is denoted by $P_n=\mathbb k[X_1,\dots ,X_n]$.

To obtain the relation $X_{i+1}T_i = T_iX_{i}+1$ from $T_iX_i T_i=qX_{i+1}$ in the nondegenerate setting, we substitute $X_i$ with the new generator $\overline{X_i}  = \dfrac{1-X_i}{1-q}$ into both sides of $T_iX_i T_i=qX_{i+1}$: 
\[ 
\begin{aligned}  
T_i X_i  T_i &= T_i \left( 1-(1-q)\overline{X_i} \right) T_i =  T_i^2-(1-q)T_i\overline{X_i} T_i
= q-(1-q)T_i - (1-q)T_i\overline{X_i}T_i \\ 
&= q(1-(1-q)\overline{X_{i+1}}) = q-q(1-q)\overline{X_{i+1}} 
\end{aligned} 
\]  
since $T_i^2 = q-(1-q)T_i$. 
This means 
\[ 
 T_i + T_i\overline{X_i}T_i = q\overline{X_{i+1}}.   
\] 
Multiply on the right by $T_i$ and let $q$ approach $1$ to obtain: 
\[ 
 1 + T_i\overline{X_i}  = \overline{X_{i+1}}T_i, 
\] 
where we have applied the relation $T_i^2=1$ on the left-hand side.

\section{Weak $\mathcal{I}_{\mathfrak{g}}(\mathfrak{sl}_2)$ and $\mathcal{I}_{\mathfrak{g}}(\mathfrak{sl}_2)$-categorification}\label{section:weak-categorifications}

Let $e,f,h$ be the usual basis of $\mathfrak {sl}_2(\mathbb{C})$.  In order to introduce an $\mathfrak{sl}_2$-categorification of a Verma module, we modify the definition in \cite{MR2373155} and make it applicable to Enright's category $\mathcal{I}_{\mathfrak{g}}(\mathfrak{sl}_2)$.

Let $\mathcal{C}$ be an additive category and $\mathbb{k}$ is a commutative ring. The category $\mathcal{C}$ is called 
{\em $\mathbb{k}$-linear} if the morphism sets $\Hom_{\mathcal{C}}(x,y)$ have the $\mathbb{k}$-module structure for all 
$x, y\in Obj(\mathcal{C})$ and compositions of morphisms are $\mathbb{k}$-bilinear maps. 
Throughout this section, let $\mathcal{A}$ be an Artinian and Noetherian $\mathbb{k}$-linear abelian category.

Recall that if $\mathcal A$ is an additive category, then the {\em split Grothendieck group} of $\mathcal A$ is the free abelian group $K_{\oplus}(\mathcal A)$ with generators $[A]$ where $A\in \mathcal A$ and relations $[A]=[A']+[A'']$ if $A\cong A'\oplus A''$.

\begin{definition}\label{definition:modified-weak-categ}    Let $n\geq -1$.  
A {\em weak $\mathcal{I}_{\mathfrak{g}}(\mathfrak{sl}_2)[n(n+2)]$-categorification} of a module $M$ from the category $\mathcal I_{\mathfrak{g}}(\mathfrak{sl}_2)[n(n+2)]$ consists of the data of a pair $(E,F)$ of additive   
endo-functors of  a category $\mathcal{A}$ 
such that: 
\begin{enumerate} 
\item the following identity holds 
\begin{equation}\label{pseudoadjoint}
B^2+C^2+2n(n+2)C+(n(n+2))^2I\cong B C+CB+2n(n+2)B, 
\end{equation} 
where
$$
B=(EF)^2+(FE)^2+2EF+2FE,\quad C=EF^2E+FE^2F,
$$
and $I$ is the identity functor, 
\item $-e=[-E]$ acts locally nilpotentally on the $\mathfrak{sl}_2$-representation  
$M=\mathbb{Q}\otimes K_{\oplus}(\mathcal{A})$ and $f$ acts as $[F]$, 
\item $M$ is a finitely generated  $\mathfrak{sl}_2$-module,
\item the classes of the indecomposable objects of $\mathcal{A}$ are weight vectors,  
\item $M=\displaystyle{\bigoplus_{\mu\in \mathfrak{h}^*}} M_{\mu}$ is a weight module, where $M_\mu:=\{m\in M: [e,f]m=\mu m\}$.  
\end{enumerate}
\end{definition}
\begin{remark}
Recall that $\text{Hom}(EM,N)\cong \text{Hom}(M,FN)$ for $E$ and $F$ adjoint functors, but $EM=0$ for a module $M$ corresponding to a highest weight vector, but $\text{Hom}(M,FN)\cong \mathbb C$.
\end{remark}

Thus our functors $E$ and $F$ are not adjoint for if $E$ and $F$ were adjoint, then we would have an isomorphism 
$\Hom_{\mathcal{V}_n}(FU_{-n}, U_{-n-2})\cong \Hom_{\mathcal{V}_n}(U_{-n}, EU_{-n-2})$. However, $EU_{-n-2}=0$ while $\Hom_{\mathcal{V}_n}(FU_{-n}, U_{-n-2})$ contains the identity morphism.

%

\begin{remark}
The element $f=[F]$ is required to be locally finite in \cite{MR2373155}, but we do not require this condition in Definition~\ref{definition:modified-weak-categ} since $f$ is not locally finite for Verma modules.   Moreover, the identity \eqnref{pseudoadjoint} follows from expanding out the equation
$$
(\Omega-n(n+2))^2=0
$$
by moving all summands with a negative coefficient to the right side and then replacing $e$ and $f$ by functors $E$ and $F$, respectively. 

\end{remark}

\begin{definition}\label{defn:sl2-categorification}
An {\em $\mathcal{I}_{\mathfrak{g}}(\mathfrak{sl}_2)$-categorification}  is a weak $\mathcal{I}_{\mathfrak{g}}(\mathfrak{sl}_2)$-categorification  
with the extra data of 
$X\in \End(F)$ and $T\in \End(F^2)$ such that 
\begin{enumerate}
\item $(\mathbf{1}_F T)\circ (T \mathbf{1}_F)\circ (\mathbf{1}_F T) 
= (T \mathbf{1}_F) \circ (\mathbf{1}_F T)\circ (T \mathbf{1}_F)$ in $\End(F^3)$, \label{Hecke1}
\item $(T+ \mathbf{1}_{F^2})\circ (T-q\mathbf{1}_{F^2})=0$ in $\End(F^2)$, \label{Hecke2}
\item $T\circ (\mathbf{1}_F X)\circ T = 
\begin{cases}
q X \mathbf{1}_F & \mbox{ if }  q\not=1 \\ 
X\mathbf{1}_F - T & \mbox{ if } q = 1 
\end{cases}$ in $\End(F^2)$, \label{Hecke3}
\end{enumerate}   
\end{definition} 
We do not assume that $X-a$ is locally nilpotent in Definition~\ref{defn:sl2-categorification} since 
$X-a$ is not locally nilpotent on the category $\mathcal{V}_0$.

\section{Khovanov's construction of the categorification $\mathcal H$ of the Heisenberg algebra $H_{\mathbb Z}$ }

In these sections, we will discuss Khovanov and Licata-Savage's categorification of the integral form of the Heisenberg algebra. 
\subsection{The integral form of the Heisenberg algebra}\label{subsection:Heisenberg-algebra}

A Heisenberg algebra is generated by $p_i$ and $q_i$ where $i$ is in the infinite set $I$, which satisfy the relations: 
\[ 
p_i q_j = q_j p_i + \delta_{ij}1, \hspace{4mm} 
p_i p_j=  p_j p_i, \hspace{4mm}
 q_i q_j = q_j q_i.
\]

Let $H_{\mathbb{Z}}$ be the unital ring with generators $a_n$ and $b_n$, where $n\geq 1$, that satisfy the relations: 
\[  
a_n b_m = b_m a_n + b_{m-1}a_{n-1}, \hspace{4mm} 
a_n a_m = a_m a_n, \hspace{4mm} 
b_n b_m = b_m b_n,  
\] 
where we set $a_0=b_0 = 1$ and $a_n= b_n = 0$ for $n < 0$. 
Since any products of $a_n$'s and $b_m$'s could be rewritten as a linear combination with nonnegative integer coefficients of monomials of the form: 
\begin{equation}\label{eq:Heisenberg-generators}
b_{m_1}b_{m_2}\cdots b_{m_k}a_{n_1}a_{n_2}\cdots a_{n_r},  \mbox{ where } 1\leq m_1 \leq m_2 \leq \ldots \leq m_k 
\mbox{ and } 1\leq n_1\leq n_2\leq \ldots \leq n_r, 
\end{equation}  
monomials in \eqref{eq:Heisenberg-generators} form a basis of $H_{\mathbb{Z}}$. 
Let $H=H_{\mathbb{Z}}\otimes \mathbb{C}$ be a $\mathbb{C}$-algebra. 
Writing 
\[ 
A(t) = 1 + a_1 t + a_2t^2 + \ldots, \hspace{4mm} 
B(u) = 1 + b_1 u + b_2 u^2 + \ldots, 
\] 
$a_n b_m = b_m a_n + b_{m-1}a_{n-1}$ could now be replaced with 
\[  A(t) B(u) = B(u)A(t)(1+tu). 
\] 
Now, let 
\[
\begin{aligned} 
\widetilde{A}(t) 
&= 1 + t A'(-t)A(-t)   +\ldots  \\ 
&= 1 + \widetilde{a}_1 t + \widetilde{a}_2 t^2 + \ldots. \\ 
\end{aligned} 
\] 
We see that $\widetilde{a}_1, \widetilde{a}_2, \ldots$ generate the same subalgebra generated by $a_1, a_2, \ldots$ and that 
\[ 
\widetilde{A}(t)B(u) = B(u)\widetilde{A}(t) + \dfrac{tu}{1-tu}. 
\] 
Equating the coefficients, we obtain 
\begin{equation}\label{eq:deriving-toward-Heisenberg-alg}
\widetilde{a}_n b_m = b_m \widetilde{a}_n + \delta_{n,m} 1, \hspace{4mm} 
\widetilde{a}_n \widetilde{a}_m = \widetilde{a}_m \widetilde{a}_n, \hspace{4mm} 
b_n b_m = b_m b_n. 
\end{equation}
So the algebra $H$ is isomorphic to the algebra generated by $\widetilde{a}_n$ and $b_m$, where $n,m>0$, together with the relations in \eqref{eq:deriving-toward-Heisenberg-alg}. 
This implies $H$ is isomorphic to the Heisenberg algebra and $H_{\mathbb{Z}}$ is its integral form. 

Note that the one variable Heisenberg algebra has generators $p$ and $q$ with one defining relation $pq-qp=1$, which appears as the algebra of operators in the quantization of the harmonic oscillator.

\section{Future work} 
In our second follow-up paper, we will address in more detail why we get a weak $\mathcal{I}_{\mathfrak{g}}(\mathfrak{sl}_2)$ and $\mathcal{I}_{\mathfrak{g}}(\mathfrak{sl}_2)$-categorification.

Moreover we will address the following of problems as a part of our future work. 
\begin{enumerate}[(i).]
\item 
For a positive integer $n$, one has $V_n/V_{-n-2}\cong L_n$ as $\mathfrak{sl}_2$-modules. It follows that the short exact sequence 
$$ \begin{CD}
0 @>>> V_{-n-2}  @>>> V_n  @>>> L_n @>>> 0.
 \end{CD}
 $$
of modules exists, which is called the Bernstein-Gelfand-Gelfand (BGG) resolution of the $n+1$-dimensional module $L_n$. 
We cannot use $\Omega-n(n+2)$ in the categorification construction because the functor sends all objects to the zero object, so this Casimir cannot be used to construct kernels and cokernels. Thus we need to find another module homomorphism $V_n\rightarrow L_n$. 
The canonical map for $V_n\rightarrow L_n$ is to send the highest weight vector in $V_n$ to the highest weight vector in $L_n$ such that for 
$f^i u_s \in V_s \subseteq L_n\boxtimes V_0$ and $F^iU_s \in \mathcal{V}_s \subseteq \mathcal{L}_n \boxtimes \mathcal{V}_0$ 
with $f^{s+j}u_s\mapsto F^{s+j}U_s$, and $f^{s+j}u_s\in V_{-s-2}$ for all $j\geq 1$. So our first question is stated as follows:  
\begin{center}
if categorification sends  $f^i u_s  \mapsto F^iU_s$, 
then find corresponding vector in  $L_s$  and object in $\mathcal{L}_s$ using \cite{MR2373155}. 
\end{center}
These leads us to the question of whether it makes sense to talk about a quotient category $\mathcal V_n/\mathcal V_{-n-2}$ and if so, does it have the structure of an $\mathfrak{sl}_2$-categorification that is naturally equivalent to the $\mathfrak{sl}_2$-categorification $\mathcal L_n$? 
\item As Enright in \cite{MR541329} shows that there is a short exact sequence 
$$ \begin{CD}
0 @>>>  V_r    @>>>   T_r  @>>>  V_{-r-2} @>>> 0 
 \end{CD}
 $$
of $\mathcal I_{\mathfrak{g}}(\mathfrak{sl}_2)[r(r+2)]$-modules, 
construct a short exact sequence 
$$ \begin{CD}
0 @>>> \mathcal V_r   @>>> \mathcal T_r  @>>> \mathcal  V_{-r-2}@>>> 0 
 \end{CD}
 $$
of $\mathcal I_{\mathfrak{g}}(\mathfrak{sl}_2)[r(r+2)]$-categorifications.   
Construct a functor $\mathcal{T}_r\rightarrow \mathcal{V}_{-r-2}$, and then 
$\mathcal{V}_r \cong \ker(\mathcal{T}_r\rightarrow \mathcal{V}_{-r-2})$. 
We believe that $\im (\Omega- r(r+2))\cong  \mathcal{V}_{-r-2}$. 
Since $\mathcal{T}_r$ is a monoidal, additive category, the notion of subtraction does not exist.  
 Enright in  \cite{MR541329} uses the decomposition in \eqnref{Enrightsresult} in an important way. We leave it as an open problem on how much of Enright's paper can be categorified. 
 \item   Let $m,n\in\mathbb N$.  Deligne in \cite{MR1106898}  introduces the notion of a tensor product of abelian categories and shows $\mathcal L_m\otimes \mathcal L_n$ exists and is an abelian category.  Then one could ask for an equivalence of $\mathfrak{sl}_2$-categorifications  
 \begin{equation}
 \mathcal L_m\otimes \mathcal L_n\cong \mathcal L_{m+n}\oplus \cdots \oplus \mathcal L_{|m+n|}.
 \end{equation} 
 \item Since we now have $\mathcal I_{\mathfrak{g}}(\mathfrak{sl}_2)$-categorifications of $V_s$ and $T_r$, does there exist a decomposition of categories whose decategorification leads to \eqnref{tensorproductdecomposition}: 
 \begin{align}   
L_n\otimes V_\lambda\cong \bigoplus_{r\in I'} T_{\lambda +r}\oplus\bigoplus_{s\in I'''} V_{\lambda+s}
\end{align}  
for $0\neq \lambda \in \mathbb Z$?
\item The category $\mathcal I_{\mathfrak{g}}(\mathfrak{sl}_2)$ is closed under tensoring by $L_n$, $T_r$ and $V_s$ and so it leads to the question of what are the tensor product decomposition theorems for the tensor product of two $\mathcal I_{\mathfrak{g}}(\mathfrak{sl}_2)$-categorifications, such as
 \begin{align} 
\mathcal L_n\boxtimes \mathcal T_\lambda,\quad  \mathcal T_{\lambda  }\boxtimes  \mathcal V_{\mu},\quad \mathcal V_{\lambda  }\boxtimes \mathcal V_{\mu},\quad  \mathcal T_\lambda\boxtimes  \mathcal T_\mu ?
\end{align}  
\item  How does one define the category $\text{Hom}(\mathcal T_n, \mathcal V_m)$, etc. and once one has defined this, can we prove results like adjoint associativity 
$$\text{Hom}(\mathcal L_k \boxtimes \mathcal V_n, \mathcal V_m)\cong \text{Hom}(\mathcal V_n,\text{Hom}(\mathcal L_k, \mathcal V_m))$$ 
or 
$$\text{Hom}(\mathcal L_k, \mathcal L_0)\boxtimes   \mathcal V_m\cong \text{Hom} (\mathcal L_k,\mathcal V_m)?$$ 

\item Since the category of finite dimensional $\mathfrak{sl}_2$-modules is an additive monoidal category, does the sum of the categorifications $\mathcal T_r$ and $\mathcal V_s$, $r\in\mathbb N$ and $s\in \mathbb Z$ form a module category in the sense of Viktor Ostrik (cf. \cite{MR1976459})? 
\item 
There are a number of tensor product decomposition theorems for $\mathfrak{sl}(2,\mathbb R)$-modules in \cite{MR1151617}.  Do these tensor product decomposition theorems categorify to interesting  $\mathfrak{sl}(2,\mathbb R)$-equivalence of categorifications?
\end{enumerate}

\appendix
\bibliography{KLR-algebras-Im}

\end{document}